\documentclass[10pt]{amsart}
\usepackage{amsmath,amssymb,amsfonts,amsthm,mathrsfs}
\usepackage{enumitem}
\usepackage{xcolor}
\usepackage[colorlinks=true,linkcolor=blue,citecolor=blue,urlcolor=blue]{hyperref}
\usepackage[foot]{amsaddr}
\usepackage[mathlines]{lineno}
\newcommand{\im}{\textup{i}}
\newcommand{\rev}[1]{#1}

\newtheorem{theorem}{Theorem}[section]
\newtheorem{lemma}{Lemma}[section]
\theoremstyle{plain}

\theoremstyle{definition}
\newtheorem{definition}{Definition}[section]

\newtheorem{remark}{Remark}[section]
\numberwithin{equation}{section}

\begin{document}

\title[Global Synchronization in Stuart--Landau Networks]{Global Complete Synchronization in Networks of Identical Stuart--Landau Oscillators}

\author[Chen, K.-W.]{Kuan-Wei Chen$^1$}
\address{$^1$ (Corresponding author) Meiji Institute for Advanced Study of Mathematical Sciences, Meiji University, 4-21-1 Nakano, Nakano-ku, Tokyo 164-8525, Japan.}
\email[Chen, K.-W.]{$^1$kwchen0613@meiji.ac.jp}

\author[Hsiao, T.-Y.]{Ting-Yang Hsiao$^2$}
\address{$^2$ (Corresponding author) International School for Advanced Studies (SISSA), Via Bonomea 265, 34136, Trieste, Italy.}
\email[Hsiao, T.-Y.]{$^2$thsiao@sissa.it}


\date{\today}

\begin{abstract}
We study global complete synchronization in finite networks of diffusively coupled Stuart--Landau oscillators with identical natural frequencies. The underlying graph is arbitrary, connected, and undirected. In contrast with phase-only models, the amplitudes evolve dynamically and may approach zero, where the polar phase equations become singular. We derive explicit sufficient conditions that keep every amplitude uniformly separated from zero and then show that the classical invariant-arc mechanism for identical Kuramoto oscillators survives in this amplitude-inclusive setting. Two synchronization criteria are obtained, combining amplitude bounds, phase confinement, an energy identity, and the asymptotic structure of the synchronized manifold. Under either criterion, every oscillator converges exponentially to the same periodic orbit of radius $\sqrt{\mu}$ and common frequency. The result provides a graph-level extension of the invariant-arc synchronization mechanism from identical Kuramoto networks to Stuart--Landau networks while retaining explicit control of the amplitude dynamics.
\end{abstract}

\maketitle

\section{Introduction}

This paper concerns networks of \emph{identical} oscillators. More precisely, we study diffusively coupled Stuart--Landau oscillators with a common natural frequency on an arbitrary finite connected undirected graph, and we seek explicit conditions that guarantee global complete synchronization. \rev{Chen and Shih \cite{chen2024phaselocked} studied phase-locked solutions and their attraction for two nonidentical Stuart--Landau oscillators.} Our point of departure is the synchronization theory for identical Kuramoto oscillators on networks, where the natural frequencies can be removed by passing to a co-rotating frame and the principal question becomes whether the coupling drives all phases to a common value.

This setting is distinct from the classical heterogeneous-frequency Kuramoto problem. In the latter, the natural frequencies are typically sampled from a distribution and one studies the transition from incoherence to a macroscopically coherent state that need not exhibit complete phase alignment; see \cite{acebron2005kuramoto,rodrigues2016kuramoto,dorfler2014synchronization} for surveys. \rev{Related results concern complete and partial synchronization on two- and three-group networks \cite{chen2024groups} and heteroclinic transitions toward synchrony in the identical all-to-all model \cite{dai2025transient}.} For identical oscillators on a finite graph, a classical route to complete synchronization is the invariant-arc, or half-circle, mechanism. If the initial phases lie in an open semicircle, the extremal phases remain confined and the phase diameter contracts under suitable connectivity and coupling assumptions. This mechanism and its graph-dependent refinements appear throughout the Kuramoto literature; see, for example, \cite{jadbabaie2004stability,ha2010complete,zhu2022emergence,abdalla2026expander}. The present work does not introduce a new phase-contraction principle. Its purpose is to determine conditions under which this established principle remains valid after genuine amplitude dynamics are added.

The Stuart--Landau equation is the universal cubic normal form for an oscillator near a supercritical Hopf bifurcation \cite{landau1944problem,stuart1960non}. Networks of Stuart--Landau and related limit-cycle oscillators have a long history, particularly in globally or mean-field coupled settings \cite{shiino1989synchronization,matthews1991dynamics,hakim1992dynamics,acebron2005kuramoto,ku2015dynamical}. Once amplitudes are retained, such systems can support phenomena that do not occur in a phase-only description, including amplitude death, cluster states, quasiperiodicity, and more complicated transitions. Recent work also examines synchronization in heterogeneous Stuart--Landau ensembles \cite{millan2025synchronization}. \rev{Synchronization beyond real phase dynamics has also been studied in the complexified Kuramoto model \cite{hsiao2026complexified}.} Phase-reduction theory explains how phase models may be recovered under appropriate structural and weak-coupling assumptions \cite{aronson1990amplitude,nakao2016phase,bick2024higher}, but our analysis is performed directly on the full amplitude--phase system and does not replace the amplitudes by a prescribed limit-cycle value.

The additional amplitude variables create a concrete obstruction to the usual Kuramoto argument. Writing $z_\ell=r_\ell e^{\im\theta_\ell}$ produces phase equations containing the ratios $r_k/r_\ell$. If an amplitude approaches zero, these coefficients lose uniform control and the polar phase itself ceases to be defined. Thus phase confinement alone is not enough: one must first prevent componentwise amplitude death, retain uniform lower and upper amplitude bounds, and then show that these time-dependent ratios do not destroy the contraction mechanism on the graph.

Our main results give two explicit sets of sufficient conditions under which this program closes. The first combines a small-coupling condition with an invariant bound on the total squared amplitude and a componentwise lower barrier. The second assumes that the linear growth parameter dominates the maximal graph degree and uses a smaller initial phase arc to obtain a direct lower comparison for every amplitude. Under either set of hypotheses, all amplitudes remain positive, the phase interval is invariant, and every solution converges to the synchronized periodic orbit
\[
    z_\ell(t)=\sqrt{\mu}\,e^{\im(\omega t+\vartheta)},
    \qquad \ell=1,\ldots,N,
\]
for a phase $\vartheta$ determined by the initial condition. The convergence is exponential. The result applies to every finite connected undirected graph and makes the dependence on the coupling, the graph Laplacian or maximal degree, and the initial data explicit.

The proof separates the amplitude and phase difficulties. We first establish global boundedness and persistence of nonvanishing amplitudes. An energy identity in the co-rotating frame then yields decay of the radial and angular velocities. The invariant-arc argument, formulated through the locally Lipschitz extremal phase functions, forces every subsequential phase limit to be constant on the connected graph. A maximum--minimum argument identifies the common limiting amplitude as $\sqrt{\mu}$. Finally, normal attraction of the synchronized circle upgrades convergence to an exponential rate.

The scope of the paper is deliberately focused. We do not study heterogeneous natural frequencies, mean-field synchronization transitions, or the broader bifurcation structure of amplitude-inclusive networks. Our contribution is instead a rigorous extension of graph-based complete synchronization for identical phase oscillators to a setting in which amplitudes remain dynamical and their possible collapse must be controlled.

The paper is organized as follows. Section~\ref{sec:model} introduces the Stuart--Landau network, the synchronization notions, and the amplitude estimates used below. Section~\ref{K} proves the two global exponential synchronization criteria on arbitrary connected undirected graphs. The proof of the componentwise amplitude lower bound is given in Appendix~\ref{sec:App A}.

\section{Stuart--Landau networks and amplitude control}\label{sec:model}
We consider an ensemble of $N$ diffusively coupled Stuart--Landau oscillators governed by
\begin{equation}
\dot z_\ell \;=\; (\mu + \im\omega_\ell)\,z_\ell \;-\; |z_\ell|^2 z_\ell
\;+\; c \sum_{k=1}^N a_{\ell k}\,(z_k - z_\ell), \quad \ell=1,\dots,N.
\label{eq:SL-network}
\end{equation}
Equation \eqref{eq:SL-network} preserves the global $U(1)$ phase-shift symmetry and couples the Hopf normal form to network diffusion. Writing polar coordinates $z_\ell=r_\ell e^{\im\theta_\ell}$ (for $r_\ell>0$), yields
\begin{equation} \label{eq:polar-theta}
\left\{\begin{aligned}
\dot r_\ell &= (\mu-r_\ell^2)r_\ell + c\sum_{k=1}^N a_{\ell k}\!\left(r_k\cos(\theta_k-\theta_\ell)-r_\ell\right), \\
\dot\theta_\ell &= \omega_\ell + c\sum_{k=1}^N a_{\ell k}\frac{r_k}{r_\ell}\sin(\theta_k-\theta_\ell).
\end{aligned}\right.
\end{equation}
\begin{remark}
We emphasize that whenever \eqref{eq:polar-theta} is well-defined, i.e., whenever
\[
r_\ell(t) > 0 \quad \text{for all } \ell \in \{1,\ldots,N\} \text{ and for all } t > 0,
\]
the system \eqref{eq:polar-theta} is equivalent to the original system \eqref{eq:SL-network}.
\end{remark}

\begin{remark}[Phases and unwrapping]\label{rem:unwrapped}
Throughout, we regard $\theta_\ell$ as an \emph{unwrapped} angle, i.e.\ a continuous map
$\theta_\ell:\,[0,\infty)\to\mathbb{R}$ chosen on any interval where $r_\ell(t)>0$ and
satisfying $e^{\im\theta_\ell(t)}=z_\ell(t)/r_\ell(t)$. Accordingly, phase differences
$\theta_k-\theta_\ell$ in \eqref{eq:polar-theta} are taken in $\mathbb{R}$.
When a modulo-$2\pi$ notion is intended we explicitly use the circular distance
\[
\operatorname{dist}_{\mathbb{S}^1}(\alpha,\beta)
=\min_{m\in\mathbb{Z}}|\alpha-\beta+2\pi m|\in[0,\pi].
\]
The evolution \eqref{eq:polar-theta} is invariant under $\theta_\ell\mapsto \theta_\ell+2\pi m_\ell$,
so wrapped and unwrapped conventions are dynamically equivalent as long as $r_\ell>0$.
If some $r_\ell(t)$ reaches $0$, the angle $\theta_\ell$ is undefined; statements that involve
$\theta_\ell$ are understood on intervals where $r_\ell>0$, or else via the complex formulation
\eqref{eq:SL-network}. In particular, to avoid the division by $r_\ell$ in
\eqref{eq:polar-theta}, we either assume an anti-amplitude-death condition after some time
(e.g.\ $\inf_{t\ge T} r_\ell(t)>0$ for all $\ell$) or carry out arguments directly in the complex
coordinates.
\end{remark}

\paragraph*{Parameters and standing assumptions.}
\begin{itemize}[leftmargin=*]
\item \textbf{Local dynamics.} $\mu>0$ is the Hopf parameter. For an uncoupled unit ($c=0$) the origin is unstable and there is a stable limit cycle of radius $\sqrt{\mu}$ with natural frequency $\omega_\ell$. The cubic coefficient has been non-dimensionalized to $1$.
\item \textbf{Frequencies.} $\omega_\ell\in\mathbb{R}$ denotes the natural frequency of node $\ell$. The identical frequency refers to $\omega_\ell\equiv\omega$, in which one may, if convenient, pass to a rotating frame to set the common frequency to zero.
\item \textbf{Coupling strength.} $c\ge 0$ scales the diffusive interaction. We keep $c$ separate from the topology so that $A=[a_{\ell k}]$ encodes structure while $c$ is a scalar control.

\item \textbf{Unweighted, undirected, connected adjacency matrix.}
We work on a simple undirected graph $G=(V,E)$ with $V=\{1,\dots,N\}$ and adjacency matrix $A=[a_{\ell k}]$ satisfying:
\begin{enumerate}[label=(\roman*),leftmargin=*]
\item \textbf{Binary edges:} $a_{\ell k}\in\{0,1\}$ for all $\ell,k$, and $a_{\ell\ell}=0$ (no self-loops).
\item \textbf{Symmetry (undirected):} $a_{\ell k}=a_{k\ell}$ for all $\ell,k$.
\item \textbf{Connectivity:} $G$ is connected; equivalently, with $D=\mathrm{diag}(d_1,\dots,d_N)$, $d_\ell=\sum_k a_{\ell k}$, and the combinatorial Laplacian,
\begin{align} \label{L=D-A}
    L=D-A,
\end{align}
we have \(\ker L = \mathrm{span}\{\mathbf{1}\}\); that is, the algebraic connectivity satisfies \(\lambda_2(L)>0\), where \(\mathbf{1}:=(1,\ldots,1)^{\mathsf T}\).
\end{enumerate}
With this notation the coupling can be written compactly as
\begin{align} \label{-cL}
c\sum_{k=1}^N a_{\ell k}(z_k-z_\ell)\;=\;-\,c\sum_{k=1}^N L_{\ell k}\,z_k,
\end{align}
so that diffusion acts along Laplacian modes.
\end{itemize}

\begin{remark}
The system is equivariant under the global phase shift $z_\ell\mapsto e^{\im\varphi}z_\ell$, hence phase locking is understood modulo a common rotation. In polar coordinates the angle $\theta_\ell$ is defined only for $r_\ell>0$; statements that involve $\theta_\ell$ are interpreted on the set where $r_j$ stays positive (or by continuity through the complex form). All symbols are dimensionless unless otherwise stated.
\end{remark}
\begin{remark} \label{remark 24}
    For the unweighted, undirected complete graph $K_N$ with $V=\{1,\dots,N\}$ and adjacency matrix $A$, we have
    \begin{align*}
        L=D-A=(N-1)I-(J-I)=N I-J,
    \end{align*}
where $I$ is identity matrix and $J_{\ell k}=1$ for all $\ell,k\in\{1,\ldots,N\}$. Therefore, we have
\begin{align*}
    0=\lambda_1(L)<\lambda_2(L)=\ldots=\lambda_N(L)=N.
\end{align*}
\end{remark}

The primary objective of this paper is to identify explicit sufficient conditions that guarantee uniform exponential complete synchronization in the identical natural-frequency case on general connected undirected graphs.
For the convenience of subsequent analysis, we introduce several notations and auxiliary functions below. Let
\begin{align*}
z(t) := (z_1(t), \ldots, z_N(t)) \in \mathbb{C}^N,
\end{align*}

\begin{align*}
   r(t) := (r_1(t), \ldots, r_N(t)) \in \mathbb{R}_{\ge 0}^N, \quad \theta(t) := (\theta_1(t), \ldots, \theta_N(t)) \in \mathbb{R}^N,
\end{align*}
for \( t \ge 0 \), and let
\[
\Omega := (\omega_1, \ldots, \omega_N) \in \mathbb{R}^N
\]
denote the vector of natural frequencies.

Throughout, we consider solutions $z(t)$ of the original
Stuart--Landau network \eqref{eq:SL-network}.
In fact, under suitable initial data and coupling strength, each amplitude remains strictly positive for all $t \ge 0$; that is,
\(r_\ell(t) = |z_\ell(t)| > 0\) for all \(\ell \in \{1,\ldots,N\}\) (see Lemma~\ref{antideath}).
Hence the polar representation \(z_\ell(t) = r_\ell(t)e^{\im\theta_\ell(t)}\) is globally well-defined,
and system \eqref{eq:polar-theta} is equivalent to \eqref{eq:SL-network}.
Under this anti-amplitude-death condition, we introduce the following notions of synchronization.

\begin{definition}[Frequency-amplitude synchronization] \label{def 2.1}
We say that the network achieves \emph{frequency-amplitude synchronization} if
\[
\lim_{t \to \infty} |r_\ell(t) - r_k(t)| = 0, \qquad
\lim_{t \to \infty} |\dot r_\ell(t) - \dot r_k(t)| = 0, \qquad
\lim_{t \to \infty} |\dot\theta_\ell(t) - \dot\theta_k(t)| = 0,
\]
for all $\ell,k \in \{1,\ldots,N\}$. If the above convergence occurs at an exponential rate, we say that the system
achieves \emph{exponential frequency-amplitude synchronization}.
\end{definition}

\begin{definition}[Complete synchronization] \label{def 2.2}
If, in addition to frequency-amplitude synchronization, the phases satisfy
\[
\lim_{t \to \infty} \operatorname{dist}_{\mathbb{S}^1}(\theta_\ell(t), \theta_k(t)) = 0,
\quad \forall\, \ell,k \in \{1,\ldots,N\},
\]
then the network is said to achieve \emph{complete synchronization}.
If the convergence is exponential, we refer to it as
\emph{exponential complete synchronization}.
\end{definition}

\rev{For the classical and hybrid all-to-all Kuramoto models, phase locking and frequency synchronization are equivalent under the respective hypotheses of \cite{hsiao2026equivalence,hsiao2026hybrid}.}

\begin{remark}
A necessary condition for complete synchronization is that all oscillators share the same natural frequency, namely
\begin{align*}
\omega_\ell = \omega, \quad \ell = 1,2,\ldots,N.
\end{align*}
When this condition holds, the oscillators are said to be identical.
\end{remark}

To ensure that the polar formulation remains globally valid, we next derive explicit sufficient conditions, adapted to the present general-graph setting, that prevent any component amplitude from reaching zero. We also record an instability criterion for complete amplitude death.

\begin{lemma}[Instability of Complete Amplitude Death]\label{amplitide-death}
If $\mu > c\,\lambda_{\max}(L)$, then the complete amplitude death equilibrium $z=0$ of \eqref{eq:SL-network} is unstable.
\end{lemma}

\begin{proof}

We identify $\mathbb{C}^{N}$ with $\mathbb{R}^{2N}$. Although the map
\(z\mapsto |z|^{2}z\) is not holomorphic, it is a polynomial in the real
and imaginary parts of $z$. Hence the vector field is real analytic in a neighborhood of the
origin.

Set
\begin{align*}
    R^{2}(t):=\sum_{\ell=1}^{N}|z_{\ell}(t)|^{2}
    =\|z(t)\|_{2}^{2}.
\end{align*}
Using the network equation, the symmetry and positive semidefiniteness of
the graph Laplacian \(L\), and the fact that the frequency terms have zero
real part, we obtain
\begin{align*}
    \frac{1}{2}\frac{\mathrm d}{\mathrm dt}R^{2}
    &=
    \mu R^{2}
    -c\,\bar{z}^{\mathsf T} Lz
    -\sum_{\ell=1}^{N}|z_{\ell}|^{4} \geq
    \bigl(\mu-c\,\lambda_{\max}(L)\bigr)R^{2}
    -R^{4}.
\end{align*}
Here we used
\begin{align*}
    \bar{z}^{\mathsf T} Lz\leq \lambda_{\max}(L)\|z\|_{2}^{2}
    \quad\text{and}\quad
    \sum_{\ell=1}^{N}|z_{\ell}|^{4}
    \leq
    \left(\sum_{\ell=1}^{N}|z_{\ell}|^{2}\right)^2.
\end{align*}

Let
\begin{align*}
    \delta:=\mu-c\,\lambda_{\max}(L)>0,
    \quad
    \rho:=\sqrt{\frac{\delta}{2}}.
\end{align*}
As long as \(0<R(t)\leq \rho\), the preceding estimate gives
\begin{align*}
    \frac{\mathrm d}{\mathrm dt}R^{2}(t)
    \geq
    \delta R^{2}(t).
\end{align*}
Consequently,
\begin{align*}
    R^{2}(t)\geq R^{2}(0)e^{\delta t}
\end{align*}
for as long as the trajectory remains in the ball
$\{z\in\mathbb{C}^{N}:\|z\|_{2}\leq\rho\}$.
Thus every nonzero initial condition arbitrarily close to the origin
eventually leaves this ball. Therefore $z=0$ is not Lyapunov stable.
\end{proof}

However, guaranteeing that each $|z_\ell|=r_\ell$ remains strictly nonzero throughout the evolution requires a more refined description. In fact, the absence of amplitude death in every component of the system (so that \eqref{eq:polar-theta} remains well-defined) depends not only on the system parameters but also on the initial conditions. We present sufficient conditions ensuring the persistence of nonvanishing amplitudes below.

\begin{lemma}[Persistence of Nonvanishing Amplitudes]\label{antideath}
Assume that
\begin{align}\label{c star}
    0<c<c^*
    :=\frac{2\mu}{3\sqrt{3N}\,\lambda_{\max}(L)}.
\end{align}
Let \(r^*\) denote the smaller positive root of
\begin{align}\label{definition r star}
    \mu x-x^3
    =
    c\sqrt{N\mu}\,\lambda_{\max}(L).
\end{align}
Suppose that the initial data satisfy
\begin{equation}\label{initial R}
\begin{aligned}
    \sum_{\ell=1}^{N}r_\ell^2(0)<N\mu,\quad r_\ell(0)>r^*, \quad\ell=1,\ldots,N.
\end{aligned}
\end{equation}
Then the solution \(z(t)\) of \eqref{eq:SL-network} exists globally and
satisfies
\begin{align}\label{antideath bounds}
    \sum_{\ell=1}^{N}r_\ell^2(t)<N\mu,
    \quad
    r_\ell(t)>r^*,
    \quad
    t\geq0,\quad \ell=1,\ldots,N.
\end{align}
Consequently, no component undergoes amplitude death, the polar
representation \(z_\ell=r_\ell e^{\im\theta_\ell}\) is globally
well-defined, and \eqref{eq:polar-theta} is equivalent to
\eqref{eq:SL-network} for all \(t\geq0\). Moreover,
\begin{align}\label{limsup amplitude bound}
    \limsup_{t\to\infty}
    \max_{1\leq\ell\leq N}r_\ell(t)
    \leq\sqrt{\mu}.
\end{align}
\end{lemma}

We provide the proof of Lemma \ref{antideath} in Appendix \ref{sec:App A}. Next, we consider the network \eqref{eq:SL-network} with $\omega_\ell \equiv 0$ for all $\ell=1,\dots,N$.
To quantify the total squared frequency of the oscillators, we introduce the energy-type functional
\begin{equation}\label{H-def}
\mathcal{H}(t) := \int_0^t \sum_{\ell=1}^N \Big( |\dot{r}_\ell(s)|^2 + |r_\ell(s)\dot{\theta}_\ell(s)|^2 \Big)\,ds.
\end{equation}
Multiplying the first equation of \eqref{eq:polar-theta} by $\dot{r}_\ell$ and the second equation by $r_\ell^2\dot{\theta}_\ell$, summing over $\ell=1,\dots,N$, and integrating in time over $[0,t]$, we obtain
\begin{equation}\label{energy-balance}
\mathcal{H}(t)
=  \int_0^t \mathcal{I}(s)\,ds + \int_0^t \mathcal{II}(s)\,ds,
\end{equation}
where
\begin{align}
\mathcal{I}(t) &:= \sum_{\ell=1}^N \Big((\mu-r_\ell^2)r_\ell - c d_\ell r_\ell\Big)\dot{r}_\ell, \label{LambdaN-def}\\[4pt]
\mathcal{II}(t) &:= c\sum_{\ell=1}^N\sum_{k=1}^N a_{\ell k}\Big(r_k\dot{r}_\ell\cos(\theta_k-\theta_\ell) + r_k r_\ell \sin(\theta_k-\theta_\ell)\dot{\theta}_\ell\Big). \label{HN-def}
\end{align}
Exploiting the symmetry of $a_{jk}$, $\mathcal{I}$, $\mathcal{II}$ can be rewritten as
\begin{equation*}\label{HN-sym}
\begin{aligned}
\mathcal{I}(t)&=\frac{d}{d t}\sum_{\ell=1}^N\left(\frac{\mu}{2}r^2_{\ell}-\frac{1}{4}r^4_{\ell}-\frac{c}{2}d_{\ell}r^2_{\ell}\right), \\
\mathcal{II}(t)&= \frac{c}{2}\,\frac{d}{dt}\sum_{\ell=1}^N\sum_{k=1}^N \Big(a_{\ell k}\,r_\ell r_k\cos\!\big(\theta_k(t)-\theta_\ell(t)\big)\Big).
\end{aligned}
\end{equation*}

If $r_\ell$ remains bounded, then the right-hand side of
\eqref{energy-balance} is bounded by a constant independent of time.
Moreover,
\[
    \dot{\mathcal H}(t)
    =
    \sum_{\ell=1}^{N}
    \left(
        |\dot r_\ell(t)|^2+
        |r_\ell(t)\dot\theta_\ell(t)|^2
    \right)
    =
    \sum_{\ell=1}^{N}|\dot z_\ell(t)|^2.
\]
Since the vector field in \eqref{eq:SL-network} is polynomial, boundedness
of $z(t)$ implies boundedness of both $\dot z(t)$ and
$\ddot z(t)=DF(z(t))\dot z(t)$. Hence $\dot{\mathcal H}$ is uniformly
continuous. The boundedness of $\mathcal H$ gives
$\dot{\mathcal H}\in L^1(0,\infty)$, and Barbalat's lemma therefore yields
\[
    \lim_{t\to\infty}\dot{\mathcal H}(t)=0.
\]
Equivalently,
\[
    \lim_{t\to\infty}\sum_{\ell=1}^N
    \Big(
        |\dot{r}_\ell(t)|^2+
        |r_\ell(t)\dot{\theta}_\ell(t)|^2
    \Big)=0.
\]
This yields the following lemma.

\begin{lemma}[Energy Functional]\label{EF}
Consider the network \eqref{eq:SL-network} with $\omega_\ell\equiv 0$ for all $\ell=1,\ldots,N$.
Let $(r(t),\theta(t))$ be a solution of \eqref{eq:polar-theta}. Assume that there exist constants $0<m<M<\infty$ such that
\begin{align*}
    m\leq r_\ell(t)\leq M,
    \quad
    t\geq0,\quad \ell=1,\ldots,N.
\end{align*}
Then
\begin{align*}
    \lim_{t\to\infty}\dot{r}_\ell=0
    \quad\text{and}\quad
    \lim_{t\to\infty}\dot{\theta}_\ell=0.
\end{align*}
\end{lemma}

\section{Global complete synchronization on connected graphs} \label{K}

In this section, we consider a general network topology given by a simple, undirected, and connected graph $G = (V,E), V = \{1,\dots,N\}$. The associated adjacency matrix $A = [a_{\ell k}]$ satisfies $a_{\ell k} \in \{0,1\}, a_{\ell\ell} = 0, a_{\ell k} = a_{k\ell}$ for all $\ell,k \in V$.

\begin{theorem}[Complete synchronization under an amplitude barrier]\label{thm:identical-complete}
Let \eqref{c star} and \eqref{initial R} hold.
Assume that $\omega_\ell = \omega \in \mathbb{R}$ for all $\ell \in \{1,\ldots,N\}$.
Let $z$ be a solution of \eqref{eq:SL-network} with initial condition satisfying
\[
 0 < \theta_\ell(0) < \pi,
 \quad
 \forall\, \ell \in \{1,\ldots,N\}.
\]
Then the solution exists globally and the network achieves \emph{complete
synchronization}. More precisely, there exists a phase
$\vartheta_0\in(0,\pi)$ such that
\begin{equation}
\begin{aligned}
&\lim_{t\to\infty}r_\ell(t)=\sqrt{\mu},
\qquad
&&\lim_{t\to\infty}
|\dot r_\ell(t)-\dot r_k(t)|=0,
\\
&\lim_{t\to\infty}
|\dot\theta_\ell(t)-\dot\theta_k(t)|=0,
\qquad
&&\lim_{t\to\infty}
\operatorname{dist}_{\mathbb{S}^1}
\bigl(
    \theta_\ell(t)-\omega t,\vartheta_0
\bigr)=0,
\end{aligned}
\end{equation}
for all $\ell,k\in\{1,\ldots,N\}$.
Equivalently,
\begin{align*}
    \lim_{t\to\infty}
    \left|
        z_\ell(t)
        -
        \sqrt{\mu}\,
        e^{\im(\omega t+\vartheta_0)}
    \right|
    =0,
    \quad
    \ell=1,\ldots,N.
\end{align*}
Moreover, all the above convergences occur at an \emph{exponential rate}; hence the network attains
\emph{exponential complete synchronization}.
\end{theorem}

\begin{proof}
By Lemma~\ref{antideath}, we may restrict attention to system~\eqref{eq:polar-theta}.
Passing to the co-rotating frame $\theta_\ell \mapsto \theta_\ell-\omega t$ gives
\begin{equation}\label{eq:polar-theta 2nd}
\left\{
\begin{aligned}
\dot r_\ell &= (\mu-r_\ell^2)r_\ell
+ c\sum_{k=1}^N a_{\ell k}\!\left(r_k\cos(\theta_k-\theta_\ell)-r_\ell\right),\\
\dot\theta_\ell &= c\sum_{k=1}^N a_{\ell k}\frac{r_k}{r_\ell}\sin(\theta_k-\theta_\ell).
\end{aligned}
\right.
\end{equation}

Define
\begin{align*}
    \theta_{\max}(t):=\max_{1\leq\ell\leq N}\theta_\ell(t),
    \quad
    \theta_{\min}(t):=\min_{1\leq\ell\leq N}\theta_\ell(t).
\end{align*}
Since they are respectively the maximum and minimum of finitely many
continuously differentiable functions, they are locally Lipschitz and
hence differentiable almost everywhere. \\

\noindent\emph{Step 1: Invariant phase interval.}
{
Let
\[
    T_*:=\sup\left\{
        T>0:
        0<\theta_{\min}(t)\leq\theta_{\max}(t)<\pi
        \ \text{for all }0\leq t<T
    \right\}.
\]
By the initial assumptions and continuity, \(T_*>0\). For almost every
\(t\in(0,T_*)\), define the active index sets
\[
    I_{\max}(t):=
    \{\ell:\theta_\ell(t)=\theta_{\max}(t)\},
    \qquad
    I_{\min}(t):=
    \{\ell:\theta_\ell(t)=\theta_{\min}(t)\}.
\]
At every point of differentiability of the extremal functions,
\[
    \dot\theta_{\max}(t)
    =
    \max_{\ell\in I_{\max}(t)}\dot\theta_\ell(t),
    \qquad
    \dot\theta_{\min}(t)
    =
    \min_{\ell\in I_{\min}(t)}\dot\theta_\ell(t).
\]
For each \(\ell\in I_{\max}(t)\), we have
\(-\pi<\theta_k(t)-\theta_\ell(t)\leq0\), and therefore
\[
    \dot\theta_\ell(t)
    =
    c\sum_{k=1}^{N}a_{\ell k}
    \frac{r_k}{r_\ell}
    \sin(\theta_k-\theta_\ell)
    \leq0.
\]
Thus \(\dot\theta_{\max}(t)\leq0\) almost everywhere on \((0,T_*)\).
Similarly, \(\dot\theta_{\min}(t)\geq0\) almost everywhere. Since the
extremal functions are locally Lipschitz, it follows that
\[
    0<\theta_{\min}(0)
    \leq\theta_{\min}(t)
    \leq\theta_{\max}(t)
    \leq\theta_{\max}(0)<\pi,
    \qquad 0\leq t<T_*.
\]
If \(T_*<\infty\), continuity gives the same strict bounds at \(T_*\),
so the interval can be extended beyond \(T_*\), a contradiction.
Consequently, \(T_*=\infty\), and
\begin{equation}\label{eq:angle-bounds}
    0<\theta_{\min}(t)\leq\theta_{\max}(t)<\pi,
    \qquad t\geq0.
\end{equation}
}

\noindent\emph{Step 2: Asymptotic alignment.}
By monotonicity, the limits
\[
\theta_{\max}(t)\downarrow\phi_M, \quad
\theta_{\min}(t)\uparrow\phi_m,
\quad 0<\phi_m\le\phi_M<\pi
\]
exist.  Lemmas~\ref{antideath} and~\ref{EF} imply $|\dot\theta_\ell|\to0$ for all $\ell$.
Selecting $t_n=n\in\mathbb{N}$, the sequence $\{(r(t_n),\theta(t_n))\}$ is bounded in $\mathbb{R}^{2N}$ by~\eqref{antideath bounds} and~\eqref{eq:angle-bounds}. Hence, by the Bolzano--Weierstrass theorem,
there exists a subsequence, still denoted by $t_n$, such that $r(t_n)\to r^\infty$ and
$\theta(t_n)\to\theta^\infty$. Choose $j$ such that
\begin{align*}
    \theta_j^\infty
    =
    \max_{1\leq\ell\leq N}\theta_\ell^\infty.
\end{align*}
Since \(\dot\theta_j(t)\to0\), passing to the limit in the second
equation of \eqref{eq:polar-theta 2nd} gives
\begin{align*}
    0
    =
    c\sum_{k=1}^{N}a_{jk}
    \frac{r_k^\infty}{r_j^\infty}
    \sin(\theta_k^\infty-\theta_j^\infty).
\end{align*}
For every $k$, $-\pi
    <
    \theta_k^\infty-\theta_j^\infty
    \leq0$, and hence every term in the above sum is nonpositive. It follows that $\sin(\theta_k^\infty-\theta_j^\infty)=0$ for every neighbor \(k\) of \(j\). Therefore, $\theta_k^\infty=\theta_j^\infty$ for every neighbor $k$ of $j$. By the connectedness of the graph,
this equality propagates along graph paths, and hence
\[
    \theta_1^\infty=\cdots=\theta_N^\infty.
\]
Since
\[
    \min_\ell\theta_\ell^\infty=\phi_m,
    \quad
    \max_\ell\theta_\ell^\infty=\phi_M,
\]
we conclude that
\[
    \phi_m=\phi_M.
\]

\noindent\emph{Step 3: Amplitude saturation.}

We now prove that
\[
    r_\ell(t)\longrightarrow\sqrt{\mu},
    \quad
    \ell=1,\ldots,N.
\]
Let \(t_n\to\infty\) be arbitrary. By \eqref{antideath bounds},
after passing to a subsequence, there exists
\(r^\infty\in(0,\infty)^N\) such that $r(t_n)\longrightarrow r^\infty$. By Step 2, $\theta_\ell(t_n)\longrightarrow\vartheta_0$
for all \(\ell\), and Lemma~\ref{EF} gives $\dot r_\ell(t_n)\longrightarrow0$. Passing to the limit in the first equation of
\eqref{eq:polar-theta 2nd}, we obtain
\[
    0
    =
    \bigl(\mu-(r_\ell^\infty)^2\bigr)r_\ell^\infty
    +
    c\sum_{k=1}^{N}a_{\ell k}
    \bigl(r_k^\infty-r_\ell^\infty\bigr),
    \qquad
    \ell=1,\ldots,N.
\]

Let
\[
    R_\infty:=\max_{1\leq\ell\leq N}r_\ell^\infty
\]
and choose \(j\) such that \(r_j^\infty=R_\infty\). Then
\begin{align*}
    \sum_{k=1}^{N}a_{jk}
    (r_k^\infty-R_\infty)\leq0,
\end{align*}
and hence $(\mu-R_\infty^2)R_\infty\geq0$. Since \(R_\infty>0\), it follows that $R_\infty\leq\sqrt{\mu}$.

Similarly, let
\[
    r_\infty:=\min_{1\leq\ell\leq N}r_\ell^\infty
\]
and choose \(j\) such that \(r_j^\infty=r_\infty\). Then
\[
    \sum_{k=1}^{N}a_{jk}
    (r_k^\infty-r_\infty)\geq0,
\]
and therefore $(\mu-r_\infty^2)r_\infty\leq0$. Since \(r_\infty>0\), we obtain $r_\infty\geq\sqrt{\mu}$. Consequently,
\[
    r_\infty=R_\infty=\sqrt{\mu},
\]
and hence
\[
    r_1^\infty=\cdots=r_N^\infty=\sqrt{\mu}.
\]
Since every convergent subsequence has the same limit, we conclude that
\[
    r_\ell(t)\longrightarrow\sqrt{\mu},
    \qquad
    \ell=1,\ldots,N.
\]
\noindent\emph{Step 4: Exponential convergence.}
Set
\[
    w_\ell(t):=e^{-\im\omega t}z_\ell(t),
    \qquad \ell=1,\ldots,N.
\]
Then \(w=(w_1,\ldots,w_N)\) satisfies
\[
    \dot w_\ell
    =
    (\mu-|w_\ell|^2)w_\ell-c(Lw)_\ell.
\]
The synchronized equilibria form the compact invariant manifold
\[
    \mathcal M
    :=
    \left\{
        \sqrt{\mu}\,e^{\im\vartheta}\mathbf 1:
        \vartheta\in\mathbb R/2\pi\mathbb Z
    \right\}
    \cong \mathbb S^1.
\]
By the global phase-shift symmetry, it suffices to linearize at
\(\sqrt{\mu}\mathbf 1\). Writing a perturbation as
\(a+\im b\), with \(a,b\in\mathbb R^N\), the linearized system is
\[
    \dot a=-(2\mu I+cL)a,
    \qquad
    \dot b=-cLb.
\]
Since the graph is connected,
\[
    0=\lambda_1(L)<\lambda_2(L).
\]
Consequently, the only zero eigenvalue is simple and corresponds to
the direction \(\im\mathbf 1\), which is tangent to \(\mathcal M\).
All eigenvalues normal to \(\mathcal M\) have real parts bounded above
by
\[
    -\gamma_0,
    \qquad
    \gamma_0:=\min\{2\mu,c\lambda_2(L)\}>0.
\]
Thus \(\mathcal M\) is a compact normally attracting invariant
manifold.

By the local stable-foliation theorem for normally hyperbolic invariant
manifolds; see, for example,
\cite[Section~4]{hirsch1977invariant}, there exists a neighborhood
\(\mathcal U\) of \(\mathcal M\) which is foliated by local stable
fibers. Trajectories on each stable fiber converge exponentially to
the corresponding base point in \(\mathcal M\).

Steps~2--3 show that
\[
    w(t)\longrightarrow
    w_\infty
    :=
    \sqrt{\mu}\,e^{\im\vartheta_0}\mathbf 1.
\]
Hence \(w(t)\in\mathcal U\) for all sufficiently large \(t\). The base
point of the stable fiber containing \(w(t)\) must be \(w_\infty\),
since \(w(t)\to w_\infty\). Therefore, for every
\(0<\gamma<\gamma_0\), there exists \(C>0\) such that
\[
    \|w(t)-w_\infty\|_{\mathbb C^N}
    \leq
    Ce^{-\gamma t},
    \qquad t\geq0,
\]
where \(C\) has been enlarged, if necessary, to include a bounded
initial time interval. Equivalently,
\[
    \max_{1\leq\ell\leq N}
    \left|
        z_\ell(t)
        -
        \sqrt{\mu}\,
        e^{\im(\omega t+\vartheta_0)}
    \right|
    \leq
    Ce^{-\gamma t}.
\]

It remains to verify the exponential convergence of the radial and
angular velocities. Denote the vector field in the co-rotating frame
by
\[
    F_\ell(w)
    :=
    (\mu-|w_\ell|^2)w_\ell-c(Lw)_\ell.
\]
Since \(F(w_\infty)=0\) and \(F\) is smooth, the preceding estimate
implies
\[
    |\dot w_\ell(t)|
    =
    |F_\ell(w(t))-F_\ell(w_\infty)|
    \leq
    Ce^{-\gamma t}.
\]
Moreover, since \(r_\ell(t)=|w_\ell(t)|\) remains uniformly bounded
away from zero, we have
\[
    \dot r_\ell
    =
    \frac{\operatorname{Re}
    \bigl(\overline{w_\ell}\dot w_\ell\bigr)}
    {|w_\ell|}
\]
and
\[
    \dot\theta_\ell-\omega
    =
    \frac{\operatorname{Im}
    \bigl(\overline{w_\ell}\dot w_\ell\bigr)}
    {|w_\ell|^2}.
\]
Therefore,
\[
    |\dot r_\ell(t)|
    +
    |\dot\theta_\ell(t)-\omega|
    \leq
    Ce^{-\gamma t},
\]
after enlarging \(C\) once more. In particular,
\[
    |\dot r_\ell(t)-\dot r_k(t)|
    +
    |\dot\theta_\ell(t)-\dot\theta_k(t)|
    \leq
    Ce^{-\gamma t}
\]
for every \(\ell,k\in\{1,\ldots,N\}\). Hence all convergences stated
in Theorem~\ref{thm:identical-complete} occur at an exponential rate.
The proof is complete.
\end{proof}

\begin{theorem}[Complete synchronization under degree dominance]\label{cor:identical-complete}
Assume that $c>0$, $\omega_\ell=\omega\in\mathbb{R}$ for all $\ell\in\{1,\ldots,N\}$ and
\begin{equation} \label{mu bigger than c d}
    \mu > c\,d_{\max}:= c\,\max_{1\le \ell\le N} \left(\sum_{k=1}^N a_{\ell k}\right).
\end{equation}
Let $z$ be a solution of \eqref{eq:SL-network} with initial condition satisfying
\[
  r_{\ell}(0)>0, \qquad 0<\theta_\ell(0)<\frac{\pi}{2}
    \quad
\forall\, \ell\in\{1,\ldots,N\}.
\]
Then the solution exists globally and the network achieves complete
synchronization. More precisely, there exists a phase
\(\vartheta^\ast\in(0,\frac{\pi}{2})\) such that
\begin{equation}
\begin{aligned}
&\lim_{t\to\infty}r_\ell(t)=\sqrt{\mu},
\qquad
&&\lim_{t\to\infty}
|\dot r_\ell(t)-\dot r_k(t)|=0,
\\
&\lim_{t\to\infty}
|\dot\theta_\ell(t)-\dot\theta_k(t)|=0,
\qquad
&&\lim_{t\to\infty}
\operatorname{dist}_{\mathbb{S}^1}
\bigl(
    \theta_\ell(t)-\omega t,\vartheta^\ast
\bigr)=0,
\end{aligned}
\end{equation}
for all \(\ell,k\in\{1,\ldots,N\}\).
Equivalently,
\begin{align*}
    \lim_{t\to\infty}
    \left|
        z_\ell(t)
        -
        \sqrt{\mu}\,
        e^{\im(\omega t+\vartheta^\ast)}
    \right|
    =0,
    \qquad
    \ell=1,\ldots,N.
\end{align*}
Moreover, all the above convergences occur at an \emph{exponential rate}; hence the network attains
\emph{exponential complete synchronization}.
\end{theorem}

\begin{proof}

Let \(T_*>0\) be the maximal time such that
\[
    r_\ell(t)>0,
    \qquad
    t\in[0,T_*),\quad \ell=1,\ldots,N.
\]
On \([0,T_*)\), introduce the co-rotating phases
\[
    \varphi_\ell(t):=\theta_\ell(t)-\omega t.
\]
Following the argument of Step~1 in the proof of
Theorem~\ref{thm:identical-complete}, we obtain
\begin{equation}\label{eq:angle-bounds 2nd}
    0
    <
    \varphi_{\min}(0)
    \leq
    \varphi_{\min}(t)
    \leq
    \varphi_{\max}(t)
    \leq
    \varphi_{\max}(0)
    <
    \frac{\pi}{2},
    \qquad
    0\leq t<T_*.
\end{equation}
In particular,
\[
    \cos(\varphi_k(t)-\varphi_\ell(t))>0
\]
for all \(k,\ell\) and \(t\in[0,T_*)\).

Set
\[
    \alpha:=\mu-cd_{\max}>0.
\]
For every \(\ell=1,\ldots,N\), we have
\begin{align*}
    \dot r_\ell
    &=
    (\mu-r_\ell^2)r_\ell
    +
    c\sum_{k=1}^{N}a_{\ell k}
    \left(
        r_k\cos(\varphi_k-\varphi_\ell)-r_\ell
    \right)\\
    &\geq
    (\mu-cd_\ell-r_\ell^2)r_\ell\\
    &\geq
    (\alpha-r_\ell^2)r_\ell.
\end{align*}
By comparison with
\[
    \dot y=(\alpha-y^2)y,
    \qquad
    y(0)=r_\ell(0),
\]
we obtain
\[
    r_\ell(t)
    \geq
    \min\left\{
        r_\ell(0),\sqrt{\alpha}
    \right\},
    \qquad
    0\leq t<T_*.
\]
Consequently, with
\[
    m_0:=
    \min_{1\leq\ell\leq N}
    \min\left\{
        r_\ell(0),\sqrt{\mu-cd_{\max}}
    \right\},
\]
we have
\[
    r_\ell(t)\geq m_0>0,
    \qquad
    0\leq t<T_*.
\]

We next show that \(T_*=\infty\). Let
\[
    R^2(t):=\sum_{\ell=1}^{N}|z_\ell(t)|^2.
\]
Since \(L\) is positive semidefinite,
\[
    \frac12\frac{\mathrm d}{\mathrm dt}R^2
    \leq
    \mu R^2-\frac1N R^4.
\]
Hence \(R(t)\) remains bounded and the solution of
\eqref{eq:SL-network} exists globally. If \(T_*<\infty\), then
\[
    |z_\ell(T_*)|
    =
    \lim_{t\uparrow T_*}r_\ell(t)
    \geq m_0>0.
\]
Thus the polar representation can be continued beyond \(T_*\),
contradicting the maximality of \(T_*\). Therefore \(T_*=\infty\).

The preceding estimates provide constants \(0<m<M<\infty\) such that
\[
    m\leq r_\ell(t)\leq M,
    \qquad
    t\geq0,\quad \ell=1,\ldots,N.
\]
Hence Lemma~\ref{EF} applies in the co-rotating frame. The arguments in
Steps~2--4 of the proof of
Theorem~\ref{thm:identical-complete} now apply verbatim. Therefore,
there exists $\vartheta^\ast\in
    \left(0,\frac{\pi}{2}\right)$ such that
\begin{align*}
    r_\ell(t)\to\sqrt{\mu},
    \qquad
    \varphi_\ell(t)\to\vartheta^\ast,
\end{align*}
and all the convergences stated in the theorem occur at an exponential
rate.
\end{proof}

\section{Conclusion}

We have established explicit sufficient conditions for global exponential complete synchronization in networks of identical Stuart--Landau oscillators on arbitrary finite connected undirected graphs. The central issue beyond the identical Kuramoto phase model is the possible loss of the polar description when an amplitude approaches zero. By combining componentwise amplitude barriers with invariant phase arcs, an energy identity, and the normal attraction of the synchronized circle, we show that the graph-based phase-contraction mechanism remains effective in the presence of evolving amplitudes. The resulting criteria identify a concrete regime in which amplitude dynamics do not obstruct complete synchronization and provide a starting point for synchronization theory on more general amplitude-inclusive networks.

\appendix
\section{Persistence of Nonvanishing Amplitudes}\label{sec:App A}
\begin{proof}[Proof of Lemma \ref{antideath}]  Define
\begin{align*}
    R^2(t):=\sum_{\ell=1}^{N}|z_\ell(t)|^2.
\end{align*}
Recalling \eqref{eq:SL-network}, \eqref{L=D-A}, and \eqref{-cL}, we
compute
\begin{align}
    \frac12\frac{\mathrm d}{\mathrm dt}R^2
    =
    \mu R^2
    -\sum_{\ell=1}^{N}|z_\ell|^4
    -c\,z^*Lz \notag\leq
    \mu R^2-\frac1N R^4,
    \label{R logistic}
\end{align}
where $z^*=\overline z^{\mathsf T}$, and we used \(z^*Lz\geq0\) and
\begin{align*}
    \sum_{\ell=1}^{N}|z_\ell|^4
    \geq
    \frac1N
    \left(\sum_{\ell=1}^{N}|z_\ell|^2\right)^2.
\end{align*}
By comparison with the scalar logistic equation
\begin{equation*}
    \dot Y
    =
    2Y\left(\mu-\frac{Y}{N}\right),
    \quad
    Y(0)=R^2(0),
\end{equation*}
and the strict inequality $R^2(0)<N\mu$, we obtain
\begin{align}\label{strict R bound}
    R^2(t)\leq Y(t)<N\mu
\end{align}
throughout the maximal interval of existence. In particular, \(z(t)\)
remains bounded, and hence the continuation theorem for ordinary
differential equations implies that the solution exists globally.

We next establish the positive lower bound. Let
\begin{align*}
    h(x):=\mu x-x^3.
\end{align*}
The function \(h\) attains its maximum at \(x=\sqrt{\mu/3}\), and
\begin{align*}
    \max_{x>0}h(x)
    =
    \frac{2\mu^{3/2}}{3\sqrt3}.
\end{align*}
By \eqref{c star},
\begin{align*}
    c\sqrt{N\mu}\,\lambda_{\max}(L)
    <
    \frac{2\mu^{3/2}}{3\sqrt3}.
\end{align*}
Therefore, \eqref{definition r star} has exactly two positive roots, and
its smaller root satisfies
\begin{align*}
    0<r^*<\sqrt{\frac{\mu}{3}}.
\end{align*}

For every $\ell$, at every time for which \(r_\ell(t)>0\), we have
\begin{equation}\label{component lower}
\begin{aligned}
    \frac12\frac{\mathrm d}{\mathrm dt}r_\ell^2
    &=\mu r_\ell^2-r_\ell^4
    -c\,\Re\!\left((Lz)_\ell\overline{z_\ell}\right)\\
    &\geq
    \mu r_\ell^2-r_\ell^4
    -c\,r_\ell |(Lz)_\ell|\\
    &\geq
    r_\ell
    \left(
        \mu r_\ell-r_\ell^3
        -c\,\lambda_{\max}(L)R
    \right),
    \end{aligned}
\end{equation}
where
\[
    |(Lz)_\ell|
    \leq
    \|Lz\|_2
    \leq
    \|L\|_2\|z\|_2
    =
    \lambda_{\max}(L)R.
\]

Suppose, by contradiction, that one of the amplitudes reaches \(r^*\).
Since \(r_\ell(0)>r^*\) for every \(\ell\), the first-contact time
\[
    \tau
    :=
    \inf\left\{
        t>0:
        \min_{1\leq\ell\leq N}r_\ell(t)=r^*
    \right\}
\]
would be positive and finite. Choose \(j\) such that
\[
    r_j(\tau)=r^*.
\]
Since \(r_j(t)>r^*\) for \(0\leq t<\tau\), the first-contact property
gives
\[
    \left.
    \frac{\mathrm d}{\mathrm dt}r_j^2(t)
    \right|_{t=\tau}
    \leq0.
\]
On the other hand, \eqref{strict R bound} and
\eqref{component lower} imply
\begin{align*}
    \left.
    \frac12\frac{\mathrm d}{\mathrm dt}r_j^2(t)
    \right|_{t=\tau}
    &>
    r^*
    \left(
        \mu r^*-(r^*)^3
        -c\sqrt{N\mu}\,\lambda_{\max}(L)
    \right)=0,
\end{align*}
where we used \eqref{definition r star}. This is a contradiction.
Therefore,
\[
    r_\ell(t)>r^*,
    \quad
    t\geq0,\quad \ell=1,\ldots,N.
\]
In particular, no amplitude vanishes, and the polar system
\eqref{eq:polar-theta} is globally well-defined.

Finally, define
\[
    M(t):=\max_{1\leq \ell\leq N}r_\ell(t).
\]
Since \(M\) is the maximum of finitely many continuously differentiable
functions, it is locally Lipschitz and hence differentiable almost
everywhere. At every point of differentiability,
\[
    \dot M(t)
    =
    \max_{\ell\in I(t)}\dot r_\ell(t),
    \quad
    I(t):=\{\ell:r_\ell(t)=M(t)\}.
\]
For every \(\ell\in I(t)\), the first equation of
\eqref{eq:polar-theta} yields
\begin{align*}
    \dot r_\ell
    &=
    (\mu-M^2)M
    +
    c\sum_{k=1}^{N}a_{\ell k}
    \left(
        r_k\cos(\theta_k-\theta_\ell)-M
    \right)\leq
    (\mu-M^2)M,
\end{align*}
because
\[
    r_k\cos(\theta_k-\theta_\ell)
    \leq r_k\leq M.
\]
Therefore,
\[
    \dot M(t)
    \leq
    (\mu-M^2(t))M(t)
\]
for almost every \(t\geq0\). By comparison with the solution \(y\) of
\[
    \dot y=(\mu-y^2)y,
    \qquad
    y(0)=M(0),
\]
we obtain \(M(t)\leq y(t)\) for all \(t\geq0\). Since
\(y(t)\to\sqrt{\mu}\) as \(t\to\infty\), it follows that
\begin{align*}
    \limsup_{t\to\infty}
    \max_{1\leq\ell\leq N}r_\ell(t)
    \leq\sqrt{\mu}.
\end{align*}
This completes the proof of Lemma~\ref{antideath}.
\end{proof}

\section*{Acknowledgements}
The authors thank the associate editor and the two anonymous referees for their careful reading and constructive comments, which led to substantial improvements in the
manuscript. K.-W. Chen is supported by the Meiji Institute for Advanced Study of Mathematical Sciences, Meiji University. T.-Y. Hsiao is supported by  the European Union  ERC CONSOLIDATOR GRANT 2023 GUnDHam, Project Number: 101124921. Views and opinions expressed are however those of the authors only and do not necessarily reflect those of the European Union or the European Research Council. Neither the European Union nor the granting authority can be held responsible for them.


\begin{thebibliography}{99}

\bibitem{abdalla2026expander}

Abdalla, Pedro and Bandeira, Afonso S. and Kassabov, Martin and
Souza, Victor and Strogatz, Steven H. and Townsend, Alex,
{\em Expander graphs are globally synchronizing},
Advances in Mathematics,
vol.~488, article 110773, 2026.


\bibitem{acebron2005kuramoto}

Acebr{\'o}n, Juan A. and Bonilla, Luis L. and
Vicente, Conrad J. P. and Ritort, F{\'e}lix and Spigler, Renato,
{\em The Kuramoto model: A simple paradigm for synchronization
phenomena},
Reviews of Modern Physics,
vol.~77, pp.~137--185, 2005.


\bibitem{aronson1990amplitude}
Aronson, Donald G and Ermentrout, G Bard and Kopell, Nancy.
{\em Amplitude response of coupled oscillators}, Physica D: Nonlinear Phenomena vol. 41, No. 3, pp. 403--449 1990.

\bibitem{bick2024higher}
Bick, Christian and B{\"o}hle, Tobias and Kuehn, Christian.
{\em Higher-order network interactions through phase reduction for oscillators with phase-dependent amplitude}, Journal of Nonlinear Science vol. 34, No. 4, pp. 77 2024.

\begingroup
\color{black}
\hypersetup{citecolor=blue,linkcolor=blue,urlcolor=blue}
\bibitem{chen2024phaselocked}
Chen, Kuan-Wei and Shih, Chih-Wen,
{\em Phase-Locked Solutions of a Coupled Pair of Nonidentical Oscillators},
Journal of Nonlinear Science, vol.~34, article 14, 2024.

\bibitem{chen2024groups}
Chen, Shih-Hsin and Hsia, Chun-Hsiung and Hsiao, Ting-Yang,
{\em Complete and Partial Synchronization of Two-Group and Three-Group Kuramoto Oscillators},
SIAM Journal on Applied Dynamical Systems, vol.~23, no.~3, pp.~1720--1765, 2024.

\bibitem{dai2025transient}
Dai, Jia-Yuan and Fiedler, Bernold and L\'opez-Nieto, Alejandro,
{\em Transient rebellions in the Kuramoto oscillator: Morse--Smale structural stability and connection graphs of finite 2-shift type},
arXiv preprint arXiv:2512.02937, 2025.
\par\endgroup

\bibitem{dorfler2014synchronization}
D{\"o}rfler, Florian and Bullo, Francesco, {\em Synchronization in complex networks of phase oscillators: A survey}, Automatica, vol. 50, no. 6, pp. 1539--1564, 2014.

\bibitem{ha2010complete}

Ha, Seung-Yeal and Ha, Taeyoung and Kim, Jong-Ho,
{\em On the complete synchronization of the Kuramoto phase model},
Physica D: Nonlinear Phenomena,
vol.~239, pp.~1692--1700, 2010.

\bibitem{hakim1992dynamics}

Hakim, Vincent and Rappel, Wouter-Jan,
{\em Dynamics of the globally coupled complex
Ginzburg--Landau equation},
Physical Review A,
vol.~46, pp.~R7347--R7350, 1992.


\bibitem{hirsch1977invariant}

Hirsch, Morris W. and Pugh, Charles C. and Shub, Michael,
{\em Invariant Manifolds: Lecture Notes in Mathematics}, vol.~583, 1977.


\begingroup
\color{black}
\hypersetup{citecolor=blue,linkcolor=blue,urlcolor=blue}
\bibitem{hsiao2026complexified}
Hsiao, Ting-Yang and Lo, Yun-Feng and Wang, Winnie,
{\em Synchronization in the complexified Kuramoto model},
Nonlinearity, vol.~39, no.~1, article 015003, 2026.

\bibitem{hsiao2026equivalence}
Hsiao, Ting-Yang and Lo, Yun-Feng and Zhu, Chengbin,
{\em Equivalence of synchronization states in the Kuramoto flow: a unified framework},
Journal of Physics: Complexity, vol.~7, no.~2, article 025010, 2026.

\bibitem{hsiao2026hybrid}
Hsiao, Ting-Yang and Lo, Yun-Feng and Zhu, Chengbin,
{\em Synchronization Relations in the Hybrid Kuramoto Flow: Equivalence and a High-Coherence Criterion},
arXiv preprint arXiv:2512.02986v3, 2026.
\par\endgroup

\bibitem{jadbabaie2004stability}

Jadbabaie, Ali and Motee, Nader and Barahona, Mauricio,
{\em On the stability of the Kuramoto model of coupled nonlinear
oscillators},
Proceedings of the 2004 American Control Conference,
vol.~5, pp.~4296--4301, 2004.


\bibitem{ku2015dynamical}

Ku, Wei L. and Girvan, Michelle and Ott, Edward,
{\em Dynamical transitions in large systems of mean-field coupled
Landau--Stuart oscillators: Extensive chaos and cluster states},
Chaos,
vol.~25, article 123122, 2015.


\bibitem{landau1944problem}
Landau, Lev D.
{\em On the problem of turbulence}, booktitle: Dokl. Akad. Nauk USSR vol. 44, pp. 311 1944.

\bibitem{matthews1991dynamics}

Matthews, Paul C. and Mirollo, Renato E. and Strogatz, Steven H.,
{\em Dynamics of a large system of coupled nonlinear oscillators},
Physica D: Nonlinear Phenomena,
vol.~52, pp.~293--331, 1991.


\bibitem{millan2025synchronization}
{ Mill{\'a}n, Ana P. and Poyato, David and Reynolds, David N. and Tudisco, Francesco, {\em Synchronization of coupled Stuart-Landau oscillators: How heterogeneity can facilitate synchronization}, arXiv preprint arXiv:2510.05243, 2025.}

\bibitem{nakao2016phase}
Nakao, Hiroya.
{\em Phase reduction approach to synchronisation of nonlinear oscillators}, Contemporary Physics vol. 57, No. 2, pp. 188--214 2016.

\bibitem{rodrigues2016kuramoto}
Rodrigues, Francisco A and Peron, Thomas K DM and Ji, Peng and Kurths, J{\"u}rgen, {\em The Kuramoto model in complex networks}, Physics Reports, vol. 610, pp. 1--98 2016.

\bibitem{shiino1989synchronization}

Shiino, Masatoshi and Frankowicz, Marek,
{\em Synchronization of infinitely many coupled limit-cycle type
oscillators},
Physics Letters A,
vol.~136, pp.~103--108, 1989.


\bibitem{stuart1960non}
Stuart, John Trevor.
{\em On the non-linear mechanics of wave disturbances in stable and unstable parallel flows Part 1. The basic behaviour in plane Poiseuille flow}, Journal of Fluid Mechanics vol. 9, No. 3, pp. 353--370 1960.

\bibitem{zhu2022emergence}

Zhu, Tingting,
{\em Emergence of synchronization in Kuramoto model with frustration
under general network topology},
Networks and Heterogeneous Media,
vol.~17, pp.~255--291, 2022.

\end{thebibliography}
\end{document}